\documentclass[a4paper,11pt]{amsart}
\usepackage{amsmath,amsfonts,amssymb,amsthm,enumerate}
\usepackage{mathrsfs}
\usepackage[pdftex]{graphicx}
\usepackage[all]{xy}
\usepackage{tikz}
\usepackage{float}

\theoremstyle{plain}
\newtheorem{thm}{Theorem}[section]
\newtheorem{lem}[thm]{Lemma}
\newtheorem{prop}[thm]{Proposition}
\newtheorem{cor}[thm]{Corollary}

\theoremstyle{definition}
\newtheorem{defn}{Definition}[section]

\theoremstyle{remark}
\newtheorem*{rem}{Remark}

\newcommand{\x}{^\times}

\DeclareMathOperator{\pd}{pd}

\title[A bound for the Milnor sum of projective plane curves]{A bound for the Milnor sum of projective plane curves in terms of GIT}
\author{Jaesun Shin}
\date{}

\address{Department of Mathematical Sciences, KAIST, 291 Daehak-ro, Yuseong-gu, Daejon 305-701, Korea}

\email{jsshin1991@kaist.ac.kr}

\begin{document}
\maketitle
\begin{abstract}
Let $C$ be a projective plane curve of degree $d$ whose singularities are all isolated. Suppose $C$ is not concurrent lines. P\l oski proved that the Milnor number of an isolated singlar point of $C$ is less than or equal to $(d-1)^{2}-\lfloor \frac{d}{2} \rfloor$. In this paper, we prove that the Milnor sum of $C$ is also less than or equal to $(d-1)^{2}-\lfloor \frac{d}{2} \rfloor$ and the equality holds if and only if $C$ is a P\l oski curve. Furthermore, we find a bound for the Milnor sum of projective plane curves in terms of GIT. 
\end{abstract}

\section{Introduction}

Let $C=V(f)$ be a projective plane curve of degree $d$. In this paper, a plane curve C means a projective plane curve that has at most isolated singularities. Moreover, we assume that $C$ is not concurrent lines.  We assume that the base field $k$ is algebraically closed and char($k$)=0. Let $f=0$ at [0,0,1]. Then, we define its Milnor number at 0 by
\begin{gather*}
\mu_{0}(f) = \dim_{k}(O_{0}/J_{f}),
\end{gather*}
where $O_{0}$ is a function germ of $f$ at the origin (in the sense of affine chart) and $J_{f} = (\partial f/ \partial x, \partial f / \partial y)$ is the Jacobian ideal of $f$. Since $\mu_{0}(f)$ is finite if and only if the origin is an isolated singular point, the Milnor number is closely related to the local properties of isolated singular points. In fact, the Milnor number has an important topological meaning. 

\begin{prop} \cite{Le}
The Milnor number is a topological invariant for IHS (isolated hypersurface singularities).
\end{prop}

By the importance of the Milnor number for IHS, there are some critical results. One of them was proven by P\l oski which says that for a projective plane curve $C$ of degree $d$ whose singularities are all isolated, not concurrent lines, the Milnor number of an isolated singlar point of $C$ is less than or equal to $(d-1)^{2}-\lfloor \frac{d}{2} \rfloor$ with equality holds if and only if $C$ is a P\l oski curve (See \cite[Definition 1.9, 1.10]{IC}). By this result, for any given point of a projective plane curve which is not concurrent lines, we get an upper bound for the Milnor number which is useful for computing the Milnor number of a given point. Also, one of the others was that of Huh which gives an upper bound for the Milnor sum of projective hypersurfaces which are not the cone over a smooth hypersurface (see \cite[Theorem 1.1]{JH}). However, since the result of Huh applies to general cases, we can expect that a bound for the Milnor sum of projective plane curves can be reduced. So the purpose of this paper is to find an upper bound for the Milnor sum of a projective plane curve and to see how such a bound can be reduced by GIT conditions. In fact, without GIT conditions, we can get the following theorem which is one of our main results:

\begin{thm}
Let $C$ be a plane curve whose singularities are all isolated and $\deg C=d \ge 5$. Then, $\pd(C)=\lfloor \frac{d}{2} \rfloor$ if and only if C is a P\l oski curve. 
\end{thm}

Recall that the gradient map of $C=V(h)$, $grad(h): \mathbb{P}^{n} \dashrightarrow \mathbb{P}^{n}$, $[x,y,z] \mapsto [\frac{\partial h}{\partial x}, \frac{\partial h}{\partial y}, \frac{\partial h}{\partial z}]$, is a map obtained from the partial derivatives of $h$. Define the polar degree of a plane curve $C=V(h)$, denoted by $\pd(C)$, is the degree of a gradient map of $h$. There is a lemma that connects $\pd(C)$ with Milnor sum. 

\begin{lem} \cite[Proposition 2.3]{FM} \label{lem:Milnor formula}(Milnor formula) 
Let $C = V(h) \subset \mathbb{P}^{n}$ be a hypersurface with isolated singularities with $\deg(C) = d$. Then, 
\[
\pd(C) = (d-1)^{n} - \sum \mu_{p}(h),
\]
where $\mu_{p}(h)$ is the Milnor number of $h$ at $p$. 
\end{lem}

By Lemma \ref{lem:Milnor formula} and Proposition \ref{prop:general}, the Milnor sum of a plane curve is bounded above by $(d-1)^{2}-\lfloor \frac{d}{2} \rfloor$ unless it is concurrent lines. Therefore, as in the case of the Milnor number of a plane curve, the Milnor sum of a plane curve also has the same bound and the equality holds only when the curve is exactly the same case as in \cite[Theorem 1.4]{AP}.  

Finally, by using Hilbert-Mumford criterion (Theorem \ref{thm:Hilbert-Mumford}), we prove that even P\l oski curves are strictly semi-stable and odd P\l oski curves are unstable (See Proposition \ref{prop:criterion}). By the previous theorem, we expect that the polar degree can be reduced by GIT conditions. Since there are many irreducible, stable plane curves of degree $d$ with polar degree $d-1$, a bound for the Milnor sum should be less than or equal to $(d-1)^{2}-(d-1)$. However, the following theorem which is one of our main results says that for some cases, this bound is very close. 

\begin{thm} 
Let $\deg C=d \ge 5$. Then, we have the followings:
\begin{enumerate}[1)]
\item Suppose $C$ is a stable curve that has either a line or a conic as an irreducible component. Then $\sum \mu_{p} \le (d-1)^{2}-(d-2)$.
\item Let $d$ be odd. Suppose $C$ is a semi-stable curve that has either a line or a conic as an irreducible component. Then $\sum \mu_{p} \le (d-1)^{2}-(d-2)$.
\item Suppose all irreducible components of $C$ are of $\deg \ge 3$. Then $\sum \mu_{p} \le (d-1)^{2}-\lceil{\frac{2d}{3}}\rceil$. 
\end{enumerate}
\end{thm}

In Section 2, we recall Hilbert-Mumford criterion (See Theorem \ref{thm:Hilbert-Mumford}) and its application to projective plane curves. Moreover, some definitions and well-known results are mentioned. Finally, in the last section, we will prove main theorems of this paper.

\section{GIT criterion and polar degree of plane curves}

The purpose of this section is to introduce some preliminaries that are useful to prove the main theorem. From now on, we denote the polar degree of a plane curve $C$ by $\pd(C)$. First, recall that the definition of semi-stability and stability in \cite[Chapter 8]{ID}.
Let $T=G_{m}^{r}$ be a torus and let $V$ be a vector space. Then, a linear representation of $T$ splits $V$ into the direct sum of eigenspaces $V=\oplus_{\chi \in \chi(T)} V_{\chi}$, where $\chi(T)$ is a set of rational characters of $T$ and $V_{\chi} = \{v \in V : t \cdot v = \chi(t) \cdot v\}$.  Since there is a natural identification between $\chi(T)$ and $\mathbb{Z}^{r}$ of abelian groups, by identifying them, we define the weight set of $V$ by $wt(V) = \{ \chi \in \chi(T) : V_{\chi} \neq \{0\}\} \subset \mathbb{Z}^{r}$. In particular, let $\overline{wt(V)}$ = convex hull of $wt(V)$ in $\chi(T) \otimes \mathbb{R} \cong \mathbb{R}^{r}$. (See \cite[Chapter9]{ID})

\begin{thm} \cite[Theorem 9.2]{ID}(Hilbert-Mumford criterion) \label{thm:Hilbert-Mumford}
Let $G$ be a torus and let $L$ be an ample $G$-linearlized line bundle on a projective $G$-variety $X$. Then, 
\begin{enumerate}[1)]
\item $x$ is semi-stable if and only if $0 \in \overline{wt(x)}$. 
\item $x$ is stable if and only if $0 \in interior(\overline{wt(x)})$. 
\end{enumerate}
\end{thm}

Also, we can check immediately that a given projective plane curve of degree $d$ is unstable by using the following proposition. 

\begin{prop} \cite[Chapter 10]{ID} \label{prop:unstability}
A projective plane curve of degree $d$ is unstable if it has a singular point of multiplicity $> \frac{2d}{3}$.
\end{prop}

Now, we recall a P\l oski curve. 

\begin{defn} \cite[Definition 1.9]{IC}
The curve $C$ is called an even P\l oski curve if $\deg C=2n$, it has $n$ irreducible componenets that are smooth conics passing through $P$, and all irreducible components intersect each other pairwise at $P$ with multiplicity 4.
\end{defn}

\begin{figure}[H]
\centering
\begin{tikzpicture}[scale=0.8]
\draw[smooth, domain=0:6.28] plot ({cos(\x r)},{sin(\x r)});
\draw[smooth, domain=0:6.28] plot ({1.5*cos(\x r)},{0.5+1.5*sin(\x r)});
\draw[smooth, domain=0:6.28] plot ({2*cos(\x r)},{0.98+2*sin(\x r)}) 
(0,2.5) node{$\vdots$} ;
\end{tikzpicture}
\caption{An even P\l oski curve}
\end{figure}
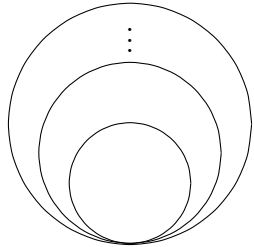

\begin{defn} \cite[Definition 1.10]{IC}
The curve $C$ is called an odd P\l oski curve if $\deg C=2n+1$, it has $n$ irreducible componenets that are smooth conics passing through $P$ and intersect each other pairwise at $P$ with multiplicity 4, and the remaining irreducible component is a line that is tangent at $P$ to all other irreducible components. 
\end{defn}

\begin{figure}[H]
\centering
\begin{tikzpicture}[scale=0.8]
\draw[smooth, domain=0:6.28] plot ({cos(\x r)},{sin(\x r)});
\draw[smooth, domain=0:6.28] plot ({1.5*cos(\x r)},{0.5+1.5*sin(\x r)});
\draw[smooth, domain=0:6.28] plot ({2*cos(\x r)},{0.98+2*sin(\x r)});
\draw[smooth, domain=0:6.28] plot ({\x-3)},{-1.05})
(0,2.5) node{$\vdots$} ;
\end{tikzpicture}
\caption{An odd P\l oski curve}
\end{figure}
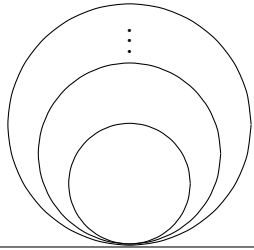

It is hard to compute the Milnor sum of a given projective plane curves directly. However, polar degree is a global one, so we can compute that more easily. So we will find a lower bound for the polar degree and use Lemma \ref{lem:Milnor formula} in order to get an upper bound for the Milnor sum of plane curves. So the problem of computing the Milnor sum of plane curves can be reduced to that of computing the polar degree. However, we can easily get the polar degree of a plane curve by the following two lemmas. 

\begin{lem} \cite[Theorem 3.1]{FM} \label{lem:2.4}
Given an irreducible curve $C \subset \mathbb{P}^{2}$ of degree $d$, we have 
\[
\pd(C) = d-1+2p_{g}+\sum(r_{p}-1),
\]
where $p_{g}$ is the geometric genus and $r_{p}$ is the number of branches at $p$. 
\end{lem}

\begin{lem} \cite[Theorem 3.1]{FM} \label{lem:polar degree formula}
Given two reduced curves $C, D$ in $\mathbb{P}^{2}$ with no common components, we have 
\[
\pd(C \cup D) = \pd(C) + \pd(D) + \sharp(C \cap D) - 1
\]
\end{lem}

The following lemma is the result of P\l oski (See \cite[Theorem 1.4]{AP}) that makes a P\l oski curve important. 
 
\begin{lem} \cite[Theorem 1.4]{AP} \label{lem:Ploski}
If $C=V(h)$ is a plane curve of degree $d \ge 5$, then $\mu_{p}(h) = (d-1)^{2} - \lfloor \frac{d}{2} \rfloor$ if and only if $C$ is a P\l oski curve and $p$ is a singular point. \end{lem}

In order to check the semi-stability of a given plane curve, we need to consider the weight set of that one. The following remark gives a way to compute the weight set for plane curves. 

\medskip

\begin{rem} \cite[Chapter 10]{ID} ($wt$ for plane curves) 
Let $Pol_{d}(E)$ be the space of degree $d$ homogeneous polynomial on $E$, where $E$ is a finite dimensional vector space. Let the standard torus $G_{m}^{2}$ act on $V=Pol_{d}(k^{3})$ via its natural homomorphism $G_{m}^{2} \rightarrow SL_{3}, (t_{1}, t_{2}) \mapsto (a_{ij})_{1 \le i,j \le 3}$, where $a_{11}=t_{1}, a_{22}=t_{2}, a_{33}=t_{1}^{-1}t_{2}^{-1}, a_{ij}=0$ for all $i \neq j$, i.e. $(t_{1}, t_{2}) \cdot x^{i}y^{j}z^{k} = t_{1}^{i-k}t_{2}^{j-k}x^{i}y^{j}z^{k}$, $i+j+k=d$. Let $V(h)$ be a plane curve of degree $d$, i.e. $i+j+k=d$, i.e. $(i-k, j-k)=(2i+j-d, 2j+i-d)$. So $wt = \{(2i+j-d, 2j+i-d) \in \mathbb{Z}^{2} : i, j \ge 0, i+j \le d, a_{ijk} \neq 0\}$. Moreover, by considering $\mathbb{R} \otimes \mathbb{Z}^{2}$, define $\bar{wt}$ by the closure of $wt$ in $\mathbb{R}^{2}$. 
\end{rem}

\section{Main result}

Now, we are ready to prove our main theorems of this paper. For notational convenience, let $r_{p}$ be the number of branches at $p$ as in Lemma \ref{lem:polar degree formula}.

\begin{lem} \label{lem:first} P\l oski curves are of polar degree $\lfloor \frac{d}{2} \rfloor$, where $\deg = d$. 
\end{lem}

\begin{proof}
First, we consider an even P\l oski curve, i.e. $d=2n$. Let $C=C_{1} \cdots C_{n}$ be an even P\l oski curve, where $C_{i}$'s are conics. Then, $\pd(C) = \pd(C_{1}) + \cdots + \pd(C_{n}) + \sharp(C_{1} \cap C_{2}) + \cdots + \sharp((C_{1} \cdots C_{n-1} \cap C_{n}) - (n-1) = n$. Next, we consider an odd P\l oski curve, i.e. $d=2n+1$. Let $C=lC_{1} \cdots C_{n}$, where $l$ is a tangent line, $C_{i}$'s are conics. Then, $\pd(C) = \pd(l) + \pd(C_{1} \cdots C_{n}) + \sharp(l \cap C_{1} \cdots C_{n}) -1 = n$. 
\end{proof}

\begin{lem} \label{lem:pre}
Let $C= C_{1} \cdots C_{m}C_{m+1} \cdots C_{k}$ be a plane curve of degree $2n$ (respectively, $2n+1$) with $m \ge 1$, $k>n$ (respectively, $k>n+1$), where $C_{1}, \cdots ,C_{m}$ are irreducible, singular plane curves and $C_{m+1} \cdots C_{k}$ is concurrent lines. Then, $\pd(C) \ge n$. 
\end{lem}

\begin{proof}
Let $l_{i} = \deg C_{i}$. Clearly, $2n=\deg C=\deg (C_{1} \cdots C_{m})+\deg (C_{m+1} \cdots C_{k}) \ge 2m+k > 2m+n$, i.e. $m<\frac{n}{2}$. Then, $\pd(C) = \pd(C_{1}) + \cdots +\pd(C_{m}) + \sharp(C_{1} \cdots C_{m} \cap C_{m+1} \cdots C_{k})-1 \ge  \sum_{i=1}^{m}(l_{i}-1)+(\sum_{p}(r_{p}-1)+\sharp(C_{1} \cdots C_{m} \cap C_{m+1} \cdots C_{k}))-1 \ge ((2n-k+m)-m)+(k-m)-1 = 2n-m-1 > \frac{3n}{2} -1 \ge n-1$, i.e. $\pd(C) \ge n$. By the same argument, we can get the result when $\deg C=2n+1$ with $k>n+1$. 
\end{proof}

\begin{prop} \label{prop:general} Let $C$ be a plane curve of $\deg C=d$. Then, $\pd(C) \ge \lfloor \frac{d}{2} \rfloor$ unless $C$ is concurrent lines. 
\end{prop}

\begin{proof}
First, we consider the case when $\deg C = 2n$. If $C$ is irreducible, it is clear by Lemma \ref{lem:2.4}. So let $C = C_{1} \cdots C_{k}$, where $C_{i}$'s are irreducible plane curves and $\deg C_{i}=l_{i}$. Then, $\pd(C) \ge \sum_{i=1}^{k}(l_{i}-1)=2n-k$. So if $k \le n$, then $\pd(C) \ge n$. So let $k>n$. Then, there exists at least 2 components which are lines. So we use induction on $n$. For small $n$, we know that the result is true. (See \cite[Theorem 3.3, 3.4]{FM}.) So suppose it holds for $n-1$. Let $C=C_{1} \cdots C_{k-2}C_{k-1}C_{k}$, where $C_{k-1}, C_{k}$ are lines. Then, $\pd(C)=\pd(C_{1} \cdots C_{k-2}) + \pd(C_{k-1}C_{k})+\sharp(C_{1} \cdots C_{k-2} \cap C_{k-1}C_{k}) -1 \ge (n-1) +\sharp(C_{1} \cdots C_{k-2} \cap C_{k-1}C_{k}) -1=n-2+\sharp(C_{1} \cdots C_{k-2} \cap C_{k-1}C_{k})$ by induction hypothesis. It is enough to consider the case when $\sharp(C_{1} \cdots C_{k-2} \cap C_{k-1}C_{k})=1$. However, by B$\acute{e}$zout's Theorem, it can happen only for the following two cases: first case is when all smooth components are lines that intersect at one point, and singular, irreducible components exist, and the second case is when $C$ is concurrent lines. However, by Lemma \ref{lem:pre}, for case 1, $\pd(C) \ge n$. Therefore, $\pd(C) \ge n$ unless $C$ is concurrent lines. For $d=2n+1$, we can use the same argument to get the result. 
\end{proof}

\begin{cor} \label{cor:bound}
Let $C = V(h)$ be a plane curve of degree $d$ in $\mathbb{P}^{2}$ whose singularities are all isolated. Then, $\sum_{p}\mu_{p}(h) \le (d-1)^{2} - \lfloor \frac{d}{2} \rfloor$ unless $C$ is concurrent lines. 
\end{cor}

\begin{proof}
By Proposition \ref{prop:general} and Lemma \ref{lem:Milnor formula}, $\sum_{p}\mu_{p}(h) = (d-1)^{2} - \pd(C) \le (d-1)^{2} - \lfloor \frac{d}{2} \rfloor$. 
\end{proof}

Since the Milnor number is nonnegative, we get the following corollary. (For another proof, see \cite[Theorem 1.1]{AP})

\begin{cor}
Let $C = V(h)$ be a plane curve of degree $d$ in $\mathbb{P}^{2}$ whose singularities are all isolated. Then, for any singular points $p$, $\mu_{p}(h) \le (d-1)^{2} - \lfloor \frac{d}{2} \rfloor$ unless $C$ is concurrent lines. 
\end{cor}

\begin{thm} \label{thm:Ploski}
Let $C$ be a plane curve whose singularities are all isolated and $\deg C=d \ge 5$. Then, $\pd(C)=\lfloor \frac{d}{2} \rfloor$ if and only if C is a P\l oski curve. 
\end{thm}

\begin{proof}
We already proved the reverse direction, so we need to prove the remaining one. Let $C=C_{1} \cdots C_{k}$ of degree $d$, where $C_{i}$'s are irreducible plane curves of $\deg C_{i}=l_{i}$. Now, we consider the following 2 cases: 

Case 1) First, suppose that all irreducible components of $C$ are smooth, i.e. $C_{i}$'s are all smooth. By Lemma \ref{lem:Ploski}, it suffices to show that if $\pd(C)=\lfloor \frac{d}{2} \rfloor$, then $C$ has only one isolated singular point. So suppose that $C$ at least two isolated singular points with $\pd(C)=\lfloor \frac{d}{2} \rfloor$. First, let $d=2n$. In this case, $n=\pd(C)=\pd(C_{1})+ \cdots + \pd(C_{k}) + (\sharp(C_{1} \cap C_{2})+ \cdots + \sharp(C_{1} \cdots C_{k-1} \cap C_{k}))-(k-1) \ge \sum_{i=1}^{k}(l_{i}-1)-(k-1)+(*) = (2n-2k+1)+(*)$, where $(*)=\sharp(C_{1} \cap C_{2})+ \cdots + \sharp(C_{1} \cdots C_{k-1} \cap C_{k})$, i.e. $n \ge (2n-2k+1)+(*)$. Since $C$ has at least 2 isolated singularities and all $C_{i}'s$ are smooth, some $\sharp$ in $(*)$ should be bigger than or equal to 2, i.e. $(*) \ge k$. So $n \ge (2n-2k+1)+(*) \ge 2n-k+1$, i.e. $k \ge n+1$. It means that $C$ has at least two lines as its irreducible components. Let $C=C_{1}C_{2}C_{3} \cdots C_{k}$, where $C_{1}, C_{2}$ are lines. Now, we consider $(*)$ again. Also, by reordering, if necessary, we can let $m$ to be the maximal number such that $C_{1}, \dots C_{m}$ are lines and intersect at one point. If $m=2$, since $\sharp(C_{1} \cap C_{2})=1$ and $\sharp(C_{1}C_{2} \cap C_{3}) \ge 2, \cdots \sharp(C_{1} \cdots C_{k-1} \cap C_{k}) \ge 2$, then $(*) \ge 2k-3$. So $n \ge 2n-2$, i.e. $n \le 2$, which is a contradiction because $d \ge 5$. So $m>2$. Then, $n=\pd(C_{1} \cdots C_{m})$ $+\pd(C_{m+1} \cdots C_{k})+ \sharp(C_{1} \cdots C_{m} \cap C_{m+1} \cdots C_{k})-1$, i.e. $\pd(C_{m+1} \cdots C_{k})=(n+1)$ $- \sharp(C_{1} \cdots C_{m} \cap C_{m+1} \cdots C_{k})$. Since $\sharp(C_{1} \cdots C_{m} \cap C_{m+1} \cdots C_{k}) \ge m$ (by using the fact that all $C_{i}$'s are smooth and by B$\acute{e}$zout's Theorem) and $\pd(C_{m+1} \cdots C_{k}) \ge \lfloor \frac{2n-m}{2} \rfloor$, we get $\lfloor \frac{2n-m}{2} \rfloor \le \pd(C_{m+1} \cdots C_{k}) \le n-m+1$. If $m=2s$, then $n-s \le n-2s+1$, i.e. $s \le 1$, which is a contradiction because $m>2$. If $m=2s+1$, then $n-s-1 \le n-2s$, i.e. $s \le 1$. Since $m>2$, we only need to check when $m=3$. If $m=3$, $n \ge (2n-2k+1)+ (*) \ge 2n-3$, i.e. $n \le 3$. However, it does not happen when $d=6$ by \cite[Theorem 3.3, 3.4]{FM}. So we need to consider when $d=2n+1$. However, by the same argument, we can prove it. So we are done in the first case. 

Case 2) Suppose that C has singular irreducible components. So let $C=C_{1} \cdots C_{m}C_{m+1} \cdots C_{k}$, where $C_{1}, \dots , C_{m}$ are singular and $C_{m+1}, \dots , C_{k}$ are smooth of $\deg C_{i}=l_{i}$ and $m \ge 1$. First, let $d=2n$. In this case, $n=\pd(C) \ge \pd(C_{1})+ \cdots \pd(C_{k}) \ge \sum_{i=1}^{k}(l_{i}-1)=2n-k$,i.e. $k \ge n$. If $k>n$, then there exists at least 2 irreducible components of $C$ which are lines. Since they are smooth, we assume that $C=(C_{1} \cdots C_{m})(C_{m+1}C_{m+2} \cdots C_{k})$, where $C_{m+1}, C_{m+2}$ are lines. Let $\deg(C_{1} \cdots C_{m}) = l$, $\deg(C_{m+1} \cdots C_{k})=2n-l$. Since $C_{m+1} \cdots C_{k}$ is not a P\l oski curve, by Case 1), $\pd(C_{m+1} \cdots C_{k})$ $ > \lfloor \frac{2n-l}{2} \rfloor$. Then, if $l=2s$, $n=\pd(C) \ge (\pd(C_{1})+ \cdots \pd(C_{m}))+\pd(C_{m+1} \cdots C_{k}) > \sum_{i=1}^{m}(l_{i}-1)+n-s=s+n-m$, i.e. $m>s$. However, $2n=\deg(C_{1} \cdots C_{m})+\deg(C_{m+1} \cdots C_{k}) \ge 3m+2n-l > 2n+s$, which is a contradiction. So let $l=2s+1$. Also, $n=\pd(C)>l-m+n-s-1$, i.e. $m>s$. Then, $2n=\deg(C_{1} \cdots C_{m})+\deg(C_{m+1} \cdots C_{k}) \ge 3m+2n-l > 3s+2n-2s-1=2n+(s-1) \ge 2n$, which is a contradiction. So when $k>n$, $\pd(C) \neq n$. Finally, it remains to prove when $k=n$. Let $k=n$. Then, $C$ has at least one line component. If there exists more than two line components in $C$, we can use the same argument so that we get a contradiction. So we only need to consider when $C$ has only one line component. It is clear that $C$ must be of the form $C=C_{1}C_{2} \cdots C_{n}$, where $C_{1}$ is of degree 3, $C_{2}$ is a line, and all $C_{i}$, $i \ge 3$, are smooth conics. For convenience, let $F=C_{2}C_{3} \cdots C_{n}$. Then, $n=\pd(C)=\pd(C_{1})+\pd(F)+\sharp(C_{1} \cap F)-1$. Since $C_{i}$'s, $i \ge 2$, are all smooth, we consider the following 2 cases: 

Case 2-1) First, let $F$ be a P\l oski curve. Since irreducible singular plane curves of degree 3 are either cusps or nodal curves, we need to consider two cases. First, let $C$ be a cusp. If $k=3$, i.e. $\deg C=6$, by \cite[Theorem 3.3, 3.4]{FM}, $\pd(C)>3$. For $k \ge 4$, we can easily get that $\sharp(C_{1} \cap F) \ge 2$. So $n=\pd(C_{1})+\pd(F)+ \sharp(C_{1} \cap F) -1 \ge 2+(n-2)+2-1 $ $= n+1$, which is a contradiction. So we need to consider when $C$ is a nodal curve. Since $\pd(C_{1}) \ge 4$ \cite[Theorem 3.4]{FM}, $n=\pd(C_{1})+\pd(F)+ \sharp(C_{1} \cap F) -1 \ge 4+(n-2)+1-1 $ $\ge n+2$, which is a contradiction.

Case 2-2) Next, let $F$ be not a P\l oski curve. Then $n=\pd(C_{1})+\pd(F)+\sharp(C_{1} \cap F)-1 > 2+(n-2)+1-1=n$, which is a contradiction. 

For $d=2n+1$, we can use the same argument to get the result. Therefore, if $C$ contains singular irreducible components, $\pd(C) \neq n$. 

So by Case 1), 2), if $\pd(C)=n$ and $\deg C \ge 5$, then $C$ is a P\l oski curve. 
\end{proof}

By Hilbert-Mumford criterion, we can check the semi-stability of P\l oski curves. 

\begin{prop} \label{prop:criterion}
An even P\l oski curve is strictly semi-stable, and an odd P\l oski curve is unstable. 
\end{prop}

\begin{proof}
Let $C$ be an even P\l oski curve. By changing projective coordinate, if necessary, we may assume that $C=(x^{2}-yz+z^{2})(x^{2}-yz+2z^{2}) \cdots (x^{2}-yz+nz^{2})$. Then, any variable that has nonzero coefficient is of the form $x^{2a}(yz)^{b}z^{2(n-a-b)}=x^{2a}y^{b}z^{2n-2a-b}$, where $0 \le a, b \le n$, $a+b \le n$. So $wt=\{(4a+b-2n, 2b+2a-2n) \in \mathbb{Z}^{2} : 0 \le a, b \le n, a+b \le n\}$. Since $2b+2a-2n \le 0$, $\bar{wt}$ lies in lower half-space of $\mathbb{R}^{2}$. Also, since $(2n,0)$, $(-n, 0)$, $(-2n, -2n)$ $\in$ $wt$, $(0,0) \in \bar{wt}$, but $(0, 0) \notin \text{interior of } \bar{wt}$. Therefore, an even P\l oski curve is strictly semi-stable. 

Also, by changing projective coordinate, if necessary, we may assume that an odd P\l oski curve is of the form $C=z(x^{2}-yz+z^{2})(x^{2}-yz+2z^{2}) \cdots (x^{2}-yz+nz^{2})$. So by the similar argument, we can get $(0,0) \notin \bar{wt}$. Therefore, an odd P\l oski curve is unstable. 
\end{proof}

So we can summarize what we get. 

\begin{thm}
Let $C$ be a plane curve of degree $d \ge 5$ in $\mathbb{P}^{2}$ whose singularities are all isolated. Suppose $C$ is not concurrent lines. Then we have the followings:
\begin{enumerate}[1)]
\item When $d=2n$, $\sum \mu_{p} \le (d-1)^{2}-\lfloor{\frac{d}{2}}\rfloor$ with equality if and only if $C$ is an even P\l oski curve. \\
For semi-stable curves, $\sum \mu_{p} \le (d-1)^{2}-\lfloor{\frac{d}{2}}\rfloor$ with equality if and only if $C$ is an even P\l oski curve. \\
For stable curves, $\sum \mu_{p} \le (d-1)^{2}-\lfloor{\frac{d}{2}}\rfloor-1$ 
\item When $d=2n+1$, $\sum \mu_{p} \le (d-1)^{2}-\lfloor{\frac{d}{2}}\rfloor$ with equality if and only if $C$ is an odd P\l oski curve. \\
For semi-stable curves, $\sum \mu_{p} \le (d-1)^{2}-\lfloor{\frac{d}{2}}\rfloor-1$ \\
For stable curves, $\sum \mu_{p} \le (d-1)^{2}-\lfloor{\frac{d}{2}}\rfloor-1$ 
\end{enumerate}
\end{thm}

\begin{proof}
By Corollary \ref{cor:bound}, Lemma \ref{lem:first}, Proposition \ref{prop:criterion}, and Theorem \ref{thm:Ploski}, we can get the result. 
\end{proof}

From now on, we find a least upper bound for the Milnor sum of plane curves and that of semi-stable plane curves of even degree. So the remaining part is to lessen an upper bound for the Milnor sum of stable curves of even degree and that of (semi)-stable curves of odd degree. In order to do this, we need the following lemmas. 

\begin{lem} \label{lem:impossible conics}
Let $C$ be a plane curve of degree 2n whose all irreducible components are conics. If $\pd(C) \le 2n-1$, then $C$ is either a P\l oski curve, (*), or ($\star$), where (*), ($\star$) are conics that intersect only at two points as the following figures show. 

\begin{figure}[H]
\centering
\begin{minipage}{.5\textwidth}
  \centering
  \begin{tikzpicture}[scale=0.5]
\draw[smooth, domain=0:6.28] plot ({0.5*cos(\x r)},{sin(\x r)});
\draw[smooth, domain=0:6.28] plot ({cos(\x r)},{sin(\x r)});
\draw[smooth, domain=0:6.28] plot ({2*cos(\x r)},{sin(\x r)}) 
(1.5,0) node{$\cdots$} (-1.5,0) node{$\cdots$};
\end{tikzpicture}
  \caption{(*)}
\end{minipage}%
\begin{minipage}{.5\textwidth}
  \centering
  \begin{tikzpicture}[scale=0.5]
\draw[smooth, domain=0:6.28] plot ({cos(\x r)},{sin(\x r)});
\draw[smooth, domain=0:6.28] plot ({1+cos(\x r)},{sin(\x r)});
\draw[smooth, domain=0:6.28] plot ({-0.3+cos(\x r)},{1.5*sin(\x r)});
\draw[smooth, domain=0:6.28] plot ({1.3+cos(\x r)},{1.5*sin(\x r)}) 
(1.5,0) node{$\cdots$} (-0.5,0) node{$\cdots$};
\end{tikzpicture}
  \caption{($\star$)}
\end{minipage}
\end{figure}
\end{lem}

\begin{proof}
For convenience, we denote the curve in FIGURE 3. and the curve in FIGURE 4. by (*), ($\star$), respectively. Let $\pd(C) \le 2n-1$ and let $C$ be not a P\l oski curve. Then, we need to show that $C$ is either (*) or ($\star$). For this, we need to show that there exists no such a form in FIGURE 5., where this is 3-conics that have common tangents with two intersection points. 
\begin{figure}[H]
\centering
\begin{tikzpicture}[scale=0.4]
\draw[smooth, domain=0:6.28] plot ({0.5*cos(\x r)},{sin(\x r)});
\draw[smooth, domain=0:6.28] plot ({cos(\x r)},{sin(\x r)});
\draw[smooth, domain=0:6.28] plot ({1.5*cos(\x r)},{0.48+1.5*sin(\x r)});
\end{tikzpicture}
\caption{impossible conics}
\end{figure}
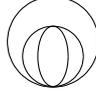
Suppose there exists such a curve. For convenience, denote (1), (2), (3) from inside to outside conics. Since (2) $\cup$ (3) is a P\l oski curve, we may let (2) to be $x^{2}-yz$, (3) to be $x^{2}-yz+z^{2}$. Clearly, (2) $\cap$ (3) = $\{[0, 1, 0]\}$ and their common tangent line at $[0, 1, 0]$ is $-z$. Since (1) is a conic, let (1) be $ax^{2}+by^{2}+cz^{2}+dxy+eyz+fzx$. Since (1) passes $[0,1,0]$, $b=0$. Now, we consider the following 2 cases: 

Case 1) First, let $a \neq 0$. We may let (1) to be $x^{2}+\alpha z^{2}+ \beta xy + \gamma yz + \delta xz$. Consider (2) $\cup$ (3) = $(x^{2}-yz)(x^{2}-yz+z^{2})$. Since the tangent line of (1) at $[0,1,0]$ is $\beta x + \gamma z$ and it has the same tangent with (2), (3), $\beta x + \gamma z = -z$, i.e. $\beta =0$, $\gamma = -1$, i.e. (1): $x^{2}+\alpha z^{2}-yz+\delta xz$. Since there exists another point in (1) $\cap$ (2) except $[0, 1, 0]$, we consider $(1) \cap (2)$, i.e. $(1)-(2) = z(\alpha z + \delta x)$. If $z=0$, then it is $[0,1,0]$. So let $\alpha z + \delta x =0$. If $\alpha =0$, then $\delta \neq 0$, i.e. (1): $x^{2}-yz+\delta xz$. However, it is easy that there exists $p \neq [0,1,0]$ such that $p \in (1) \cap (3)$ in this case, which is a contradiction. So let $\alpha \neq 0$. Also, in this case $\delta \neq 0$. (This can be obtained by the following argument: If $\delta =0$, (1): $x^{2}+\alpha z^{2} -yz$, so $(1) \cap (2) = \{[0,1,0]\}$, which is a contradiction.) Also, by some calculation, $[\frac{-(\alpha-1)}{\delta}z, \frac{(\alpha -1)^{2}}{\delta^{2}}z+z, z]$ is another common root of (1), (3). We get a contradiction again. So there exists no such a curve when $a \neq 0$. 

Case 2) Now, let $a=0$. We may let (1) to be $cz^{2}+dxy+eyz+fzx$. Since the tangent of (1) at $[0,1,0]$ is $dx+ez$, $dx+ez=-z$, i.e. $d=0$, $e=-1$, i.e. (1) is $cz^{2}-yz+fzx$. However, it is reduced, and it gives a contradiction. So there exists no such (1), i.e. we get the following: when $C_{1} \cdots C_{k}$ is a P\l oski curve but $C_{1} \cdots C_{k}C_{k+1}$ is not, if $C_{k+1}$ meet at some point of $C_{i}$ that is not a common point of $C_{1} \cdots C_{k}$, $C_{k+1}$ meet at some point of all $C_{i}$ that is not common. (It can be obtained by the following way: I proved that such (1) does not exist and also, by considering the intersection multiplicity, we can get it.) So suppose $C$ is neither a P\l oski curve, (*) nor $(\star)$. We assume that $C=C_{1} \cdots C_{k}C_{k+1} \cdots C_{n}$, where $C_{1} \cdots C_{k}$ are the maximal number of conics that forms a P\l oski curve in $C$. Since $C$ is not a P\l oski curve and a conic is a P\l oski curve, $1 \le k <n$. Then, by the above argument, $\pd(C_{1} \cdots C_{n}) \ge -k^{2}+nk+n$ (because $\sharp(C_{1} \cap C_{2})=1$, $\sharp(C_{1} \cdots C_{k-1} \cap C_{k})=1$, and $\sharp(C_{1} \cdots C_{k} \cap C_{k+1}) \ge 1+k$, $\sharp(C_{1} \cdots C_{n-1} \cap C_{n}) \ge 1+k$). Since $1 \le k < n,$ minimum occurs when $k=1,$ i.e. when $k=1$, $\pd(C) \ge 2n-1$ ,and $\pd(C) > 2n-1,$ otherwise. However, when $k=1,$ it is clear that $\pd(C)=2n-1$ if and only if $C$ is either (*) or $(\star),$ which is a contradiction. Therefore, $\pd(C) > 2n-1$ if $C$ is neither a P\l oski curve, (*), or $(\star)$. 
\end{proof}

\begin{lem} \label{lem:polar degree of conics}
Let $C$ be a stable plane curve of $\deg C=2n$ whose all irreducible components are conics. Then $\pd(C) > 2n-1$. 
\end{lem}

\begin{proof}
We use the same notation in the previous lemma. It is easy that $\pd(\text{*})=\pd(\star)=2n-1$. So we need to check the stability of (*), $(\star)$. Since (*) is $(x^{2}-yz)(x^{2}-2yz) \cdots (x^{2}-nyz)$ and $(\star)$ is $(x^{2} -yz+xz) \cdots (x^{2}-yz+nxz)$, by Hilbert-Mumford Criterion, they are strictly semi-stable. So, since a P\l oski curve, (*), and $(\star)$ are strictly-semistable, if $C$ is stable, $\pd(C) > 2n-1$. 
\end{proof}

The following lemma is an immediate consequence of the previous lemmas. 

\begin{lem} \label{lem:conics}
Let $C$ be as in Lemma \ref{lem:polar degree of conics}. Then, 
\begin{enumerate}[1)]
\item $C$ is a P\l oski curve if and only if $\pd(C) = n$ 
\item $C$ is either (*) or ($\star$) in Lemma \ref{lem:impossible conics} if and only if $\pd(C) = 2n-1$ 
\item $\pd(C) > 2n-1$, otherwise. 
\end{enumerate}
\end{lem}

Now, we are ready to get an upper bound for the Milnor sum of (semi)stable curves.

\begin{prop}
Let $C$ be a plane curve with $\deg C = d \ge 5$ that has either a line or a conic as an irreducible component. Suppose $C$ is stable. Then, $\pd(C) \ge d-2$. 
\end{prop}

\begin{proof}
We consider the following 3 cases: 

Case 1) First, let $C=C_{1} \cdots C_{m}C_{m+1} \cdots C_{k}$, where $\deg C_{i}=1$ for $1 \le i \le m$, $\deg C_{i}=2$ for $m+1 \le i \le k$. For convenience, let $D=C_{1} \cdots C_{m}$, $E=C_{m+1} \cdots C_{k}$. If $D=\emptyset$, we already proved it. So let $D \neq \emptyset$. Also, let us consider the case when $E=\emptyset$,i.e. $C=D$. By reordering, if necessary, let $D=C_{1} \cdots C_{t}C_{t+1} \cdots C_{d}$, where $t$ is the maximal number of concurrent lines in $D$ and $C_{1} \cdots C_{t}$ is concurrent lines. By the stability condition, $2 \le t \le \frac{2d}{3}$ (See Proposition \ref{prop:unstability}). Then, $\pd(D) \ge -t^{2}+(d+1)t-d$ (because $\sharp(C_{1} \cdots C_{i} \cap C_{i+1})=1$ for all $i=1, \cdots, t-1$ and $\ge t$ for all $i \ge t$). So the minium occurs when $t=2$. So $\pd(D) \ge d-2$. So we also let $E \neq \emptyset$. First, let us consider the case when $E$ is a P\l oski curve. If $D$ is nonconcurrent lines, then we can easily get the result. So suppose that $D$ is concurrent lines. If common points of $D$ and $E$ coincide, then by using $\sharp(D \cap E) \ge 1+(m-1)(k-m)$, $k \le \frac{2d}{3}$, and $2k-m=d$, we can easily get that $\pd(C) \ge d-2$. Also, if they do not coincide, it is easy that $\pd(C) \ge d-2$ when $m=1, \cdots 5$. For $m \ge 6$, by using $\sharp(D \cap E) \ge 1+(m-1)(2(k-m)-1)$, $m \le \frac{2d}{3}$, and $k=\frac{d+m}{2} \le \frac{5d}{6}$, we can get $\pd(C) \ge d-2$. If $E$ is not a P\l oski curve, by Lemma \ref{lem:conics}, we get the result. 

Case 2) Next, let $C=C_{1} \cdots C_{m}C_{m+1} \cdots C_{k}$, where $\deg C_{i}=2$ for $1 \le i \le m$, $\deg C_{i} \ge 3$ for $m+1 \le i \le k$. Let $E=C_{1} \cdots C_{m}$, $F=C_{m+1} \cdots C_{k}$. By the given condition, $E \neq \emptyset$. Also, by Lemma \ref{lem:conics}, we also let $F \neq \emptyset$, i.e. $1 \le m <k$. First, we suppose that $E$ is a P\l oski curve. So we assume that $E: (x^{2}-yz) \cdots (x^{2}-yz + (m-1)z^{2})$. We claim that $\sharp(E \cap F) \ge m+1$. If one of $C_{i}$, $m+1 \le i \le k$, does not pass $[0, 1, 0]$, we are done. So let all $C_{i}$'s pass through $[0, 1, 0]$. The case when m=1 is obtained automatically by proof of the case when $m \ge 2$. So let $m \ge 2$. Fix $m+1 \le i \le k$. Suppose $C_{i} \cap E = \{[0, 1, 0]\}$. Then, since $C_{i} \cap C_{1} = \{[0, 1, 0]\}$, $C_{i}: (x^{2}-yz)f + z^{l_{i}}$, or $(x^{2}-yz)f + x^{l_{i}}$, where $f$ is a homogeneous polynomial of degree $l_{i}-2$ in $k[x,y,z]$. Second one can be proven similarly as the first one, so we assume that $C_{i}: (x^{2}-yz)f + z^{l_{i}}$. Since $C_{i} \cap C_{2} = \{[0, 1, 0]\}$, $(x^{2}-yz)f + z^{l_{i}} = $ $z^{2}f+z^{l_{i}}=z^{2}(f+z^{l_{i}-2})$ has $z=0$ as a unique root. Since base field is algebraically closed, $f=az^{l_{i}-2}$, where $a \in k$, base field. So $C_{i}: (x^{2}-yz)(az^{l_{i}-2}) + z^{l_{i}}$, which is a contradiction since $C_{i}$ is irreducible. So $C_{i}$ has another intersetion point with $C_{1}$, which means that $I_{[0,1,0]}(C_{i} \cap C_{j}) < (\deg C_{i})(\deg C_{j}) = 2\deg C_{i}$, where $1 \le j \le m$, and $I_{[0,1,0]}(C_{i} \cap C_{j})$ is the intersection multiplicity of $C_{i}$ and $C_{j}$ at $[0, 1, 0]$. So $\sharp(C_{i} \cap C_{j}) \ge 2$, all $1 \le j \le m$. Therefore, $\sharp(E \cap F) \ge m+1$, which proves the claim. By the claim, $\sharp(E \cap C_{i}) \ge m+1$ for all $i \ge m+1$. So $\pd(C_{1} \cdots C_{k})=\pd(E)+\pd(C_{m+1})+\cdot+\pd(C_{k})+(\sharp(E \cap C_{m+1}) + \cdots + \sharp(EC_{m+1} \cdots C_{k-1} \cap C_{k}))-(k-m) \ge d-1 \ge d-2$ since $k-m \ge 1$. So let us consider the case when $E$ is not a P\l oski curve. However, by using $\sharp(E \cap C_{j}) \ge 2$ for all $j$ with $m+1 \le j \le k$ and Lemma \ref{lem:conics}, we easily get $\pd(C) \ge d-1 \ge d-2$. 

Case 3) In general, let $C=C_{1} \cdots C_{m}C_{m+1} \cdots C_{t}C_{t+1} \cdots C_{k}$, where $\deg C_{i}=1$ for $1 \le i \le m$, $\deg C_{i}=2$ for $m+1 \le i \le t$, $\deg C_{i} \ge 3$ for $t+1 \le i \le k$. For convenience, let $D=C_{1} \cdots C_{m}$, $E=C_{m+1} \cdots C_{t}$, $F=C_{t+1} \cdots C_{k}$. If $D=\emptyset$, $C$ is in Case 2), so let $D \neq \emptyset$. If $F=\emptyset$, then it is Case 1), so let $F \neq \emptyset$. Therefore, we need to deal with $E$. First, let $E \neq \emptyset$. If $D$ is concurrent lines, then $\pd(C)=\pd(D)+\pd(EF)+\sharp(D \cap EF)-1 \ge (d-m-1)+(1+(t-m)(m-1))-1 \ge d-2$ because $t-m \ge 1$ and Case 2) always holds without the stability condition. If $D$ is not concurrent lines, since $\pd(D) \ge m-2$ and $\pd(EF) \ge d-m-1$, $\pd(C) \ge d-2$. Finally, we suppose that $E=\emptyset$. If $D$ is concurrent lines, $\pd(C)=\pd(D)+\pd(C_{m+1}) + \cdots + \pd(C_{k})+(\sharp(D \cap C_{m+1}) + \cdots + \sharp(DC_{m+1} \cdots C_{k-1} \cap C_{k}))-(k-m) \ge \sum_{i=m+1}^{k}(\deg C_{i}-1)+(\sum_{p \in C_{m+1}}(r_{p}-1)+\sharp(D \cap C_{m+1}))+ \cdots + (\sum_{p \in C_{k}}(r_{p}-1)+\sharp(DC_{m+1} \cdots C_{k-1} \cap C_{k}))-(k-m) \ge d-2$. If $D$ is not concurrent lines, then by using $\sum_{p \in C_{i+1}}(r_{p}-1)+ \sharp(DC_{m+1} \cdots C_{i} \cap C_{i+1}) \ge m$ for all $i \ge m+1$, $\pd(D) \ge m-2$ and the above argument, $\pd(C) \ge (d-2)+(m-2)(k-m)$. Since $D$ is not concurrent lines and $F \neq \emptyset$, $m \ge 3$ and $k-m \ge 1$. So $\pd(C) \ge d-1 \ge d-2$.   
\end{proof}

By the previous proposition, we get a bound for the polar degree of stable curves. If $C$ is of odd degree, since a P\l oski curve is not semi-stable, by the same argument, we can get the same result for semi-stable curve $C$ as the following proposition says.

\begin{prop}
Let $C$ be a plane curve with $\deg C = d \ge 5$ that has either a line or a conic as an irreducible component, where $d$ is odd. Suppose $C$ is semi-stable. Then, $\pd(C) \ge d-2$. 
\end{prop}    
      
So we need to consider the case when all irreducible components of $C$ are of $\deg \ge 3$. The following lemma gives a better bound of such a curve. 

\begin{lem}
Let $\deg C=d \ge 5$. Suppose all irreducible components of $C$ are of $\deg \ge 3$. Then, $\pd(C) \ge \lceil{\frac{2d}{3}}\rceil$, where $\lceil{\frac{2d}{3}}\rceil$ is a round up integer of $\frac{2d}{3}$.  
\end{lem}

\begin{proof}
Let $C=C_{1} \cdots C_{m}C_{m+1} \cdots C_{k}$, where $C_{i}$'s are irreducible, plane curves with $\deg C_{i} \ge 4$ for $1 \le i \le m$, $\deg C_{j} = 3$ for $m+1 \le j \le k$. Let $D=C_{1} \cdots C_{m}$, $E=C_{m+1} \cdots C_{k}$. Then, $\pd(C) = \pd(D)+\pd(E)$ $+\sharp(D \cap E)-1 \ge \sum_{i=1}^{m}(\deg C_{i}-1)+2(k-m)=(d-3(k-m))-m+2(k-m) = d-k$. Since $3k \le d$ by degree consideration, $k \le \frac{d}{3}$. So $\pd(C) \ge d-k \ge \frac{2d}{3}$, i.e. $\pd(C) \ge \lceil{\frac{2d}{3}}\rceil$.  
\end{proof}

So we get the following result:

\begin{thm} 
Let $\deg C=d \ge 5$. Then, we have the followings:
\begin{enumerate}[1)]
\item Suppose $C$ is a stable curve that has either a line or a conic as an irreducible component. Then $\sum \mu_{p} \le (d-1)^{2}-(d-2)$.
\item Let $d$ be odd. Suppose $C$ is a semi-stable curve that has either a line or a conic as an irreducible component. Then $\sum \mu_{p} \le (d-1)^{2}-(d-2)$.
\item Suppose all irreducible components of $C$ are of $\deg \ge 3$. Then $\sum \mu_{p} \le (d-1)^{2}-\lceil{\frac{2d}{3}}\rceil$. 
\end{enumerate}
\end{thm}

\bigskip
{\em Acknowledgements}. This work is part of my master thesis. I would like to thank my advisor Yongnam Lee, for his advice, encouragement and teaching. This work was supported by Basic Science Program through the National Research Foundation of Korea funded by the Korea government(MSIP)(No.2013006431).

\bigskip

\bibliographystyle{abbrv}

\begin{thebibliography}{00}

\bibitem{IC} I. Cheltsov, Worst singularities of plane curves of given degree, preprint available at http://arxiv.org/abs/1409.6186.

\bibitem{ID} I. Dolgachev,
         {\em Lectures on Invariant Theory}, LMS Lecture Note Series 296, Cambridge University Press, Cambridge, 2003.
         
\bibitem{FM} T. Fassarella and N. Medeiros, On the polar degree of projective hypersurfaces, J.Lond. Math. Soc. (2) 86 (2012), no. 1, 259-271. 

\bibitem{JH} J. Huh, Milnor numbers of projective hypersurfaces with isolated singularities, Duke Math.J.163 (2014), no.8, 1525-1548.

\bibitem{Le} D. T. L$\hat{\text{e}}$,
         {\em Topologie des singularit$\acute{\text{e}}$s des hypersurface complexes}, $Ast\acute{e}risque$ $\bf{7/8}$ (Singularit$\acute{\text{e}}$s Carg$\grave{\text{e}}$se) 1973, pp. 171-182.
                
\bibitem{AP} A. P\l oski, A bound for the Milnor number of plane curve singularities, Cent. Eur. J. Math. 12 (2014), no.5, 688-693.   

\end{thebibliography}

\end{document}